\documentclass[12pt]{article}
\usepackage{amsmath,amssymb}
\usepackage{amsthm}
\usepackage{array}
\newtheorem{thm}{Theorem}[section]
\newtheorem{prop}[thm]{Proposition}
\newtheorem{cor}[thm]{Corollary}
\newtheorem{lemma}[thm]{Lemma}

\theoremstyle{note}

\def\a{\mathfrak a}

\def\R{\mathfrak R}

\def\F{\mathfrak F}

\def\H{\mathfrak H}

\def\rit#1{{\mbox{\rm #1}}}
\def\modx#1#2{\equiv#1\hspace{-1mm}\mod #2}
\def\nmodx#1#2{\not\equiv#1\hspace{-1mm}\mod #2}
\def\itemx#1{\item[{\rm(#1)}]}
\def\br#1{\{#1\}}

\begin{document}
\title{Singular values of generalized $\lambda$ functions}
\maketitle
\begin{center}
 N{\sc oburo} I{\sc shii}\end{center}
\section{Introduction}
 For a positive integer $N$, let $\Gamma_1(N)$ be the subgroup of $\rit{SL}_2(\mathbf Z)$ defined by
\[
\Gamma_1(N)=\left\{\left. \begin{pmatrix} a & b \\ c & d \end{pmatrix}\in \rit{SL}_2(\mathbf Z)~\right |~ a-1\equiv c \equiv 0 \mod N \right\}.
\]
We denote by $A_1(N)$ the modular function field with respect to $\Gamma_1(N)$. For a positive integer $N\geq 6$, let $\mathfrak a=[a_1,a_2,a_3]$ be a triple of integers with the properties $0<a_i\leq N/2$ and $a_i\ne a_j$ for any $i,j$. For an element $\tau$ of the complex upper half plane $\mathfrak H$, we denote by $L_\tau$ the lattice of $\mathbf C$ generated by $1$ and $\tau$ and by $\wp(z;L_\tau)$ the Weierstrass $\wp$-function relative to the lattice $L_\tau$. In \cite{II}, we defined a modular function $W_{\a}(\tau)$ with respect to $\Gamma_1(N)$ by
 \[
W_{\a}(\tau)=\frac{\wp (a_1/N;\tau)-\wp (a_3/N;\tau)}{\wp (a_2/N;\tau)-\wp (a_3/N;\tau)}.
\]
This function is one of generalized $\lambda$ functions introduced by S.Lang in Chapter 18, \S6 of \cite{LA}. He describes that it is interesting to investigate special values of generalized $\lambda$ functions at imaginary quadratic points, to see if they generate the ray class field. Here a point of $\H$ is called an imaginary quadratic point if it generates an imaginary quadratic field over $\mathbf Q$. In Theorem 3.7 of \cite{IK}, we showed, under a rather strong condition that $a_1a_2a_3(a_1-a_3)(a_2-a_3)$ is prime to $N$, that the values of $W_\a$ at imaginary quadratic points are units of ray class fields. Let $j$ be the modular invariant function. We showed in Theorem 5 of \cite{II} that each of the functions $W_{[3,2,1]},W_{[5,2,1]}$ generates $A_1(N)$ over $\mathbf C(j)$. In this article, we shall study the functions $W_\a$ in the particular case: $a_2=2,a_3=1$. To simplify the notation, henceforth we denote by $\Lambda_k$ the function $W_{[k,2,1]}$. We shall prove that if $2<k<N/2$, then $\Lambda_k$ generates $A_1(N)$ over $\mathbf C(j)$. This result implies that for an imaginary quadratic point $\alpha$ such that $\mathbf Z[\alpha]$ is the maximal order of the field $K=\mathbf Q(\alpha)$, the values $\Lambda_k(\alpha)$ and $\displaystyle e^{2\pi i/N}$ generate the ray class field of $K$ modulo $N$ over the Hilbert class field of $K$.  Let $\delta=(k,N)$ be the greatest common divisor of $k$ and $N$. On the assumption that $k$ satisfies either (i)~$\delta=1$ or (ii) $\delta>1,(\delta,3)=1$ and $N/\delta$ is not a power of a prime number, we shall prove that values of $\Lambda_k$ at imaginary quadratic points are algebraic integers. Throughout this article, we use the following notation:\newline 
For a function $f(\tau)$ and $A=\begin{pmatrix}a&b\\c&d\end{pmatrix}\in\rit{SL}_2(\mathbf Z)$, $f[A]_2,f\circ A$ represent
\[
f[A]_2=f\left(\frac{a\tau+b}{c\tau+d}\right)(c\tau+d)^{-2},~f\circ A=f\left(\frac{a\tau+b}{c\tau+d}\right).
\]
The greatest common divisor of $a,b\in\mathbf Z$ is denoted by $(a,b)$.
 For an integral domain $R$, $R((q))$ represents the ring of power series of a variable $q$ with coefficients in $R$ and $R[[q]]$ is a subring of $R((q))$ of power series with non-negative order. For elements $\alpha,\beta$ of $R$, the notation $\alpha\mid\beta$ represents that $\beta$ is divisible by $\alpha$, thus $\beta=\alpha\gamma$ for an element $\gamma\in R$.  

\section{Auxiliary results}
Let $N$ be a positive integer greater than $6$. Put $q=\rit{exp}(2\pi i\tau/N),\zeta=\exp(2\pi i/N)$. For an integer $x$, let
$\{x\}$ and $\mu (x)$ be the integers defined  by the following conditions:
\[
\begin{split}
&0\le \{x\}\le \frac N2,\quad \mu (x)=\pm 1,\\
&\begin{cases}\mu(x)=1\qquad &\text{if } x\modx {0,N/2}N,\\
             x\equiv \mu (x)\{x\} \mod N\qquad&\text{otherwise.}
\end{cases}
\end{split}
\]

For an integer $s$ not congruent to $0 \mod N$, let 
\[\phi_s(\tau)=\frac 1{(2\pi i)^2}\wp \left(\frac s N;L_\tau\right)-1/12.
\]
 Let $\displaystyle A=\begin{pmatrix}a&b\\c&d\end{pmatrix}\in\rit{SL}_2(\mathbf Z)$. Put $s^*=\mu (sc)sd,u_s=\zeta^{s^*}q^{\{sc\}}$. Then by Lemma 1 of \cite{II}, we have
{\small
\begin{equation}\label{eq1}
\phi_s[A]_2=
\begin{cases}\displaystyle
\frac{\zeta^{s^*}}{(1-\zeta^{s^*})^2}-\sum_{m=1}^{\infty}\sum_{n=1}^{\infty}n(1-\zeta^{s^*n})(1-\zeta^{-s^*n})q^{mnN}&\text{if }\{sc\}=0,\\
\displaystyle\sum_{n=1}^{\infty}n u_s^n-\displaystyle\sum_{m=1}^{\infty}\sum_{n=1}^{\infty}n(1-u_s^n)(1-u_s^{-n})q^{mnN}&\text{otherwise}.
\end{cases}
\end{equation}
}
We shall need next lemmas and propositions in the following sections.
\begin{lemma}\label{lem1}
Let $r,s,c,d$ be integers such that $0<r\ne s\leq N/2,~(c,d)=1$. Assume that $\{rc\}=\{sc\}$. Put $r^*=\mu(rc)rd, s^*=\mu(sc)sd$. Then we have $\zeta^{r^*-s^*}\ne 1$. Further if $\{rc\}=\{sc\}=0,N/2$, then $\zeta^{r^*+s^*}\ne 1$.
\end{lemma}
\begin{proof}
The assumption $\{rc\}=\{sc\}$ implies that $(\mu(rc)r-\mu(sc)s)c\modx 0N$. If $\zeta^{r^*-s^*}=1$, then $(\mu(rc)r-\mu(sc)s)d\modx 0N$. From $(c,d)=1$, we obtain $\mu(rc)r-\mu(sc)s\modx 0N$. This shows $r=s$. Suppose $\{rc\}=\{sc\}=0,N/2$ and $\zeta^{r^*+s^*}=1$. Then we have $(r+s)c\modx 0N,~(r+s)d\modx 0N$. Therefore $r+s\modx 0N$. This is impossible, because $0<r\ne s\leq N/2$.
\end{proof}
\begin{lemma} Let $k\in\mathbf Z,\delta=(k,N)$. 
\begin{enumerate}
\itemx i For an integer $\ell$, if $\delta\mid \ell$, then $(1-\zeta^\ell)/(1-\zeta^k)\in\mathbf Z[\zeta]$.
\itemx {ii} If $N/\delta$ is not a power of a prime number, then $1-\zeta^k$ is a unit of $\mathbf Z[\zeta]$.
\end{enumerate}
\end{lemma}
\begin{proof} If $\delta | \ell$, then there exist an integer $m$ such that $\ell\modx{mk}{N}$. 
Therefore $\zeta^\ell=\zeta^{mk}$ and $(1-\zeta^k)\mid (1-\zeta^\ell)$. This shows (i). Let $p_i~(i=1,2)$ be distinct prime factors of $N/\delta$. 
Since $N/p_i=\delta (N/(\delta p_i))$, $1-\zeta^\delta\mid 1-\zeta^{N/p_i}$. 
Therefore $1-\zeta^\delta \mid p_i~(i=1,2)$. This implies that $1-\zeta^\delta$ is a unit. Because of $(k/\delta,N/\delta)=1$, $1-\zeta^k$ is also a unit .
\end{proof}
From \eqref{eq1} and Lemma~\ref{lem1}, we immediately obtain the following two propositions
.
\begin{prop}\label{prop1} Let $r,s\in\mathbf Z$ such that $0<r\ne s \leq N/2$.
\begin{enumerate}
\itemx i If $\{rc\},\{sc\}\ne 0$, then 
\[
(\phi_r-\phi_s)[A]_2\equiv \sum_{n=1}^\infty n(u_r^n-u_s^n)+u_r^{-1}q^N-u_s^{-1}q^N \mod q^N\mathbf Z[\zeta][[q]].
\]
\itemx{ii} If $\{rc\}=0$ and $\{sc\}\ne 0$, then
\[(\phi_r-\phi_s)[A]_2\equiv \frac{\zeta^{rd}}{(1-\zeta^{rd})^2}-\sum_{n=1}^\infty nu_s^n-u_s^{-1}q^N \mod q^N\mathbf Z[\zeta][[q]].\]
\itemx {iii} If $\{rc\}=\{sc\}=0$, then 
\[
(\phi_r-\phi_s)[A]_2\equiv \frac{-\zeta^{sd}(1-\zeta^{(r-s)d})(1-\zeta^{(r+s)d})}{(1-\zeta^{rd})^2(1-\zeta^{sd})^2}~\mod q^N\mathbf Z[\zeta][[q]],
\]
\end{enumerate}
\end{prop}
\begin{prop}\label{prop2} Let $r,s\in\mathbf Z$ such that $0<r\ne s \leq N/2$. Put $\ell=\min(\{rc\},\{sc\})$. Then
\[(\phi_r-\phi_s)|[A]_2=\theta_{r,s}(A)q^\ell(1+qh(q)),\]
where $h(q)\in\mathbf Z[\zeta][[q]]$ and $\theta_{r,s}(A)$ is a non-zero element of $\mathbf Q(\zeta)$ given as follows.
In the case $\{rc\}=\{sc\}$,
\[
\theta_{r,s}(A)=\begin{cases}-\zeta^{s^*}(1-\zeta^{r^*-s^*})\quad&\text{if }\ell\ne 0,N/2,\\
           -\zeta^{s^*}(1-\zeta^{r^*-s^*})(1-\zeta^{r^*+s^*})\quad&\text{if }\ell=N/2,\\
\displaystyle\frac{-\zeta^{s^*}(1-\zeta^{r^*-s^*})(1-\zeta^{r^*+s^*})}{(1-\zeta^{r^*})^2(1-\zeta^{s^*})^2}\quad&\text{if }\ell=0.
\end{cases}
\]
In the case $\{rc\}\ne\{sc\}$,assuming that $\{rc\}<\{sc\}$,
\[
\theta_{r,s}(A)=\begin{cases}\displaystyle \zeta^{r^*}\quad&\text{if }\ell\ne 0,\\
\displaystyle\frac{\zeta^{r^*}}{(1-\zeta^{r^*})^2}\quad&\text{if }\ell=0.
\end{cases}
\]
\end{prop}
\section{Values of $\Lambda_k$ at imaginary quadratic points}
In this section, we shall prove that the values of $\Lambda_k=W_{[k,2,1]}$ at imaginary quadratic points are algebraic integers.
\begin{prop}\label{prop3} Let $k$ be an integer such that $3\leq k<N/2$. Put $\delta=(k,N)$. Assume either \rit{(i)} $\delta=1$ or \rit{(ii)} $\delta>1,(\delta,3)=1$ and $N/\delta$ is not a power of a prime number. Then for $A\in\rit{SL}_2(\mathbf Z)$,we have
\[\Lambda_k\circ A\in\mathbf Z[\zeta]((q)).\]
\end{prop}
\begin{proof} Put $A=\begin{pmatrix}a&b\\c&d\end{pmatrix}$. Proposition \ref{prop2} shows 
 \[\Lambda_k\circ A=\omega f(q),\]
 where $\omega=\theta_{k,1}(A)/\theta_{2,1}(A)$ and $f$ is a power series in $\mathbf Z[\zeta]((q))$. Therefore it is sufficient to prove that $\omega\in \mathbf Z[\zeta]$. First we consider the case $\{c\}\ne 0$. Let $\{2c\}\ne\{c\}$
. By (ii) of Proposition\ref{prop2}, we see $1/(\phi_2-\phi_1)[A]_2 \in \mathbf Z[\zeta]((q))$. 
Further if $\{kc\}\ne0$, then $(\phi_k-\phi_1)[A]_2\in\mathbf Z[\zeta][[q]]$. If $\{kc\}=0$,then $\delta>1$ and $c\modx0{N/\delta}$. Therefore $\zeta^{kd}$ is a primitive $N/\delta$-th root of unity. The assumption (ii) shows $1-\zeta^{kd}$ is a unit. Thus $(\phi_k-\phi_1)[A]_2\in\mathbf Z[\zeta][[q]]$. Hence we have $\omega\in\mathbf Z[\zeta]$. Let $\{2c\}=\{c\}$. 
Then, since $\{c\}\ne 0$, we have $N\modx 03,~(k,3)=1$ and $\{c\}=\{2c\}=\{kc\}=N/3$, $\mu(2c)=-\mu(c)$, $\mu(kc)=(\frac k3)\mu(c)$, where $(\frac *3)$ is the Legendre symbol. 
By the same proposition, we know that $\displaystyle\omega=(1-\zeta^{(\mu(kc)k-\mu(c))d})/(1-\zeta^{-3\mu(c)d})$. Since $\mu(kc)k-\mu(c)\modx 03$, we have $\omega\in\mathbf Z[\zeta]$. 
Next consider the case $\{c\}=0$. Then we have $\{c\}=\{2c\}=\{kc\}=0,\mu(c)=\mu(2c)=\mu(kc)=1$, $(d,N)=1$ and 
\[
\omega=\left(\frac{1-\zeta^{2d}}{1-\zeta^{kd}}\right)^2\cdot\frac{(1-\zeta^{(k-1)d})(1-\zeta^{(k+1)d})}{(1-\zeta^d)(1-\zeta^{3d})}.
\]
If $\delta =1$, then $(kd,N)=1$. If $\delta\ne 1$, then the assumption (ii) implies $(1-\zeta^{kd})$ is a unit. Therefore $(1-\zeta^{2d})/(1-\zeta^{kd})\in\mathbf Z[\zeta]$.  If $N\not\equiv 0\mod 3$, then since $(3d,N)=1$, we know \[\frac{(1-\zeta^{(k-1)d})(1-\zeta^{(k+1)d})}{(1-\zeta^d)(1-\zeta^{3d})}\in\mathbf Z[\zeta].\]
 If $N\modx 03$, then $(k,3)=1$ and one of $k+1,k-1$ is divisible by $3$. Lemma \ref{lem1} (i) gives  
\[\displaystyle\frac{(1-\zeta^{(k-1)d})(1-\zeta^{(k+1)d})}{(1-\zeta^d)(1-\zeta^{3d})}\in\mathbf Z[\zeta].\]
 Hence we obtain $\omega\in\mathbf Z[\zeta]$. 
\end{proof}
\begin{thm} Let $\alpha$ be an imaginary quadratic point. Then $\Lambda_k(\alpha)$ is an algebraic integer.
\end{thm}
\begin{proof}
Let $\R$ be a transversal of the coset decomposition of $\rit{SL}_2(\mathbf Z)$ by $\Gamma_1(N)\{\pm E_2\}$, where $E_2$ is the unit matrix. Consider a modular equation $\Phi(X,j)=\prod_{A\in \R}(X-\Lambda_k\circ A)$. Since $\Lambda_k\circ A$ has no poles in $\H$ and $\Lambda_k\circ A\in\mathbf Z[\zeta]((q))$ by Proposition \ref{prop3}, the coefficients of $\Phi(X,j)$ are polynomials of $j$ with coefficients in $\mathbf Z[\zeta]$. Since $j(\alpha)$ is an algebraic integer (see Theorem 10.23 in \cite{C1}), $\Phi(X,j(\alpha))$ is a monic polynomial with algebraic integer coefficients. Because $\Lambda_k(\alpha)$ is a root of $\Phi(X,j(\alpha))$, it is an algebraic integer. 
\end{proof}
Further we can show that $\Phi(X,j)\in \mathbf Z[j][X]$ and that $\Lambda_k(\alpha)$ belongs to the ray class field of $\mathbf Q(\alpha)$ modulo $N$. For details, see \S 3 of \cite{IK}.
\begin{cor} Let $A\in\rit{SL}_2(\mathbf Z)$. Then the values of the function $\Lambda_k\circ A$ at imaginary quadratic points are algebraic integers. In particular, the function
\[
\frac{\wp (k\tau/N;\tau)-\wp (\tau/N;\tau)}{\wp (2\tau/N;\tau)-\wp (\tau/N;\tau)}
\]
takes algebraic and integral values at imaginary quadratic points, for $2<k<N/2$. 
\end{cor}
\begin{proof}
Let $\alpha$ be an imaginary quadratic point. Then, $A(\alpha)$ is an imaginary quadratic point. Therefore, we have the former part of the assertion. If we put $\displaystyle A=\begin{pmatrix}0&-1\\1&0 \end{pmatrix}$, then from the transformation formula of $\wp((r\tau+s)/N;L_\tau)$ in \S2 of \cite{II}, we obtain the latter part.
\end{proof} 
\section{Generators of $A_1(N)$}
Let $A(N)$ be the modular function field of the principal congruence subgroup $\Gamma (N)$ of level $N$. For a subfield $\F$ of $A(N)$, let us denote by $\F_{\mathbf Q(\zeta)}$ the subfield of $\F$ consisted of all modular functions having Fourier coefficients in $\mathbf Q(\zeta)$.

\begin{thm}\label{th1} Let $k$ be an integer such that $2<k<N/2$. Then we have $A_1(N)_{\mathbf Q(\zeta)}=\mathbf Q(\zeta)(\Lambda_k,j)$ 
\end{thm}
\begin{proof} 
 By Theorem 3 of Chapter  6 of \cite{LA}, the field $A(N)_{\mathbf Q(\zeta)}$ is a Galois extension over $\mathbf Q(\zeta)(j)$ with the Galois group $\rit{SL}_2(\mathbf Z)/\Gamma(N)\{\pm E_2\}$ and the field $A_1(N)_{\mathbf Q(\zeta)}$ is the fixed field of the subgroup $\Gamma_1(N)\{\pm E_2\}$. Since $\Lambda_k\in A_1(N)_{\mathbf Q(\zeta)}$, to prove the assertion, we have only to show $A\in\Gamma_1(N)\{\pm E_2\}$, for $A\in\rit{SL}_2(\mathbf Z)$ such that $\Lambda_k\circ A=\Lambda_k$. Let $A=\begin{pmatrix}a&b\\ c&d\end{pmatrix}\in\rit{SL}_2(\mathbf Z)$ such that $\Lambda_k\circ A=\Lambda_k$.  Since the order of $q$-expansion of $\Lambda_k$ is $0$ and that of $\Lambda_k\circ A$ is $\min(\br{kc},\br{c})-\min(\br{2c},\br{c})$ by Proposition \ref{prop2}, we have 
\begin{equation}\label{eq2}
\min(\br{kc},\br{c})=\min(\br{2c},\br{c}). 
\end{equation}
By considering  power series modulo $q^N$, thus modulo $q^N\mathbf Q(\zeta)[[q]]$, from Proposition \ref{prop3} we obtain 
\begin{equation}\label{eq3}
\theta_{2,1}(E_2)(\phi_k-\phi_1)[A]_2\equiv \theta_{k,1}(E_2)(\phi_2-\phi_1)[A]_2\quad \mod q^N
\end{equation}
For an integer $i$, put $u_i=\zeta^{\mu(ic)id}q^{\br{ic}},\omega_i=\zeta^{(\mu(ic)i-\mu(c))d}$. 
First of all, we shall prove that $c\modx 0N$. Let us assume $c\nmodx 0N$.
Suppose that $\br{2c}=\br c$. Since $\br c\ne 0$, we see $\br c=N/3$. Further since by \eqref{eq2} $\br{kc}\geq \br c$, we have $(k,3)=1, \br c=\br{2c}=\br{kc}=N/3$ and $u_k=\omega_ku_1$, $u_2=\omega_2u_1$. Lemma \ref{lem1} gives that $\omega_k,\omega_2\ne 1,\omega_k\ne\omega_2$.
By \eqref{eq3} and Proposition \ref{prop1},
\[
\begin{split}
\theta_{2,1}(E_2)&\left(\sum_nn(u_k^n-u_1^n)+u_k^{-1}q^N-u_1^{-1}q^N\right)\equiv\\
&\theta_{k,1}(E_2)\left(\sum_nn(u_2^n-u_1^n)+u_2^{-1}q^N-u_1^{-1}q^N\right)\quad \mod q^N.
\end{split}
\]
Therefore
\[
\begin{split}
\theta_{2,1}(E_2)&\left(\sum_nn(\omega_k^n-1)u_1^n+(\omega_k^{-1}-1)u_1^{-1}q^N\right)\equiv\\
&\theta_{k,1}(E_2)\left(\sum_nn(\omega_2^n-1)u_1^n+(\omega_2^{-1}-1)u_1^{-1}q^N\right)\quad \mod q^N.
\end{split}
\]
Since $q^N=\zeta^{-3\mu(c)d}u_1^3$,
\[
\begin{split}
\theta_{2,1}(E_2)&((\omega_k-1)u_1+(2(\omega_k^2-1)+\zeta^{-3\mu(c)d}(\omega_k^{-1}-1)u_1^2)\equiv\\
&\theta_{k,1}(E_2)((\omega_2-1)u_1+(2(\omega_2^2-1)+\zeta^{-3\mu(c)d}(\omega_2^{-1}-1)u_1^2)\quad \mod u_1^3.
\end{split}
\]
By comparing the coefficients of $u_1,u_1^2$ on both sides, we have
\[
2(\omega_k+1)-\omega_k^{-1}\zeta^{-3\mu(c)d}=2(\omega_2+1)-\omega_2^{-1}\zeta^{-3\mu(c)d}.
\]
This equation implies that $\zeta^{3\mu(c)d}\omega_2\omega_k=-1/2$. We have a contradiction. Suppose $\br{2c}>\br c$. Then by \eqref{eq2}, we know $\br{kc}\geq\br c$. If $\br{kc}>\br c$, then the $q$-expansion of $\Lambda\circ A$ begins with $1$. Thus $\theta_{k,1}(E_2)=\theta_{2,1}(E_2)$. This gives that $(1-\zeta^{k+2})(1-\zeta^{k-2})=0$. We have a contradiction. If $\br{kc}=\br c$, then $\br{kc},\br c\ne 0,N/2$ and $u_k=\omega_ku_1$. By considering mod $q^N$ as above, we obtain
\[
\begin{split}
\theta_{2,1}(E_2)&\left(\sum_nn(\omega_k^n-1)u_1^n+(\omega_k^{-1}-1)u_1^{-1}q^N\right)\equiv\\
&\theta_{k,1}(E_2)\left(\sum_nn(u_2^n-u_1^n)+u_2^{-1}q^N-u_1^{-1}q^N\right)\quad \mod q^N.
\end{split}
\]
Thus
\[
\begin{split}
u_1+2(\omega_k+1)u_1^2-&\omega_k^{-1}u_1^{-1}q^N\equiv\\
& u_1-u_2+2u_1^2-u_2^{-1}q^N+u_1^{-1}q^N-2u_2^2+\cdots\quad\mod q^N.
\end{split}
\]
Therefore
\[2\omega_k u_1^2-(\omega_k^{-1}+1)u_1^{-1}q^N+h_1(u_1)\equiv -u_2-u_2^{-1}q^N-2u_2^2+h_2(u_2)\quad\mod q^N,\]
where $h_i(u_i)$ is a polynomial of $u_i$ with terms $u_i^n,n>2$. 
Since $\br{2c}>\br c$,we see $\br{2c}\leq N-\br{2c}<N-\br c$. Therefore we have $2\br c<N-\br c$ and $2\br c=\br{2c}=N-\br{2c}$ or $2\br c=\br{2c}<N-\br{2c}$. By comparing the coefficients of first terms, we obtain $2\omega_k\zeta^{2\mu(c)d}=-(\zeta^{\mu(2c)2d}+\zeta^{-\mu(2c)2d})$ in the case $\br{2c}=N-\br{2c}$ and  $2\omega_k\zeta^{2\mu(c)d}=-\zeta^{\mu(2c)2d}$ in the case $\br{2c}<N-\br{2c}$. In the former case, $N$ is even and $\br{2c}=N/2$. So we have $\mu(2c)2c\modx 0{N/2}$ and $\mu(2c)2d\modx 0{N/2}$. Therefore from $(c,d)=1$ we obtain $2\modx 0{N/2}$.  This is impossible. In the latter case, clearly we have a contradiction. Suppose $\br{2c}<\br c$. Then $\br{kc}=\br{2c}$. If $\br{2c}=0$, then $k,N$ are even and $\br c=N/2$. From Proposition \ref{prop1}, we get
\[
\begin{split}
(\phi_k-\phi_1)[A]_2&=\frac{\zeta^{kd}}{(1-\zeta^{kd})^2}-(\zeta^d+\zeta^{-d})q^{N/2}\quad \mod q^N,\\
(\phi_2-\phi_1)[A]_2&=\frac{\zeta^{2d}}{(1-\zeta^{2d})^2}-(\zeta^d+\zeta^{-d})q^{N/2}\quad \mod q^N.
\end{split} 
\]
By using \eqref{eq3},
\[
\begin{split}
\theta_{2,1}(E_2)\frac{\zeta^{kd}}{(1-\zeta^{kd})^2}=\theta_{k,1}(E_2)\frac{\zeta^{2d}}{(1-\zeta^{2d})^2},\\
\theta_{2,1}(E_2)(\zeta^d+\zeta^{-d})=\theta_{k,1}(E_2)(\zeta^d+\zeta^{-d}).
\end{split}
\]
If $\zeta^d+\zeta^{-d}=0$, then $2d \modx 0{N/2}$. Since $2c\modx 0{N/2}$ and $(c,d)=1$, we see $2\modx 0{N/2}$. This is impossible. Therefore $\theta_{2,1}(E_2)=\theta_{k,1}(E_2)$ and $\frac{\zeta^{kd}}{(1-\zeta^{kd})^2}=\frac{\zeta^{2d}}{(1-\zeta^{2d})^2}$. This implies that $(1-\zeta^{(k+2)d})(1-\zeta^{(k-2)d})=0$. Lemma \ref{lem1} gives a contradiction. Hence $\br{2c},\br c\ne 0,N/2$. Let $u_k=\omega u_2$, where $\omega=\omega_k/\omega_2$. By \eqref{eq3},
\[
\begin{split}
\theta_{2,1}(E_2)(\sum_n &n(\omega^nu_2^n-u_1^n)+\omega^{-1}u_2^{-1}q^N-u_1^{-1}q^N)\equiv \\
&\theta_{k,1}(E_2)(\sum_n n(u_2^n-u_1^n)+u_2^{-1}q^N-u_1^{-1}q^N)\quad \mod q^N.\end{split}
\]
Therefore $\theta_{2,1}(E_2)\omega=\theta_{k,1}(E_2)$ and
\[
\begin{split}
\sum_nn(\omega^n-\omega)u_2^n&+(\omega^{-1}-\omega)u_2^{-1}q^N\equiv \\
&\sum_nn(1-\omega)u_1^n+(1-\omega)u_1^{-1}q^N\quad \mod q^N.
\end{split}
\]
Since by Lemma \ref{lem1},$\omega\ne 1$, we have
\[
2\omega u_2^2-(1+\omega^{-1})u_2^{-1}q^N+h_2(u_2)\equiv -u_1-u_1^{-1}q^N-2u_1^2+h_1(u_1)~~\mod q^N,
\]
where $h_i(u_i)$ is a polynomial of $u_i$ with terms $u_i^n,n>2$. Since 
$\br c<N-\br c<N-\br{2c}$,we have $2\br{2c}=\br c$ and $2\omega\zeta^{2\mu(2c)2d}=-\zeta^{\mu(c)c}$. This gives a contradiction. Hence we have $c\modx 0N$.
Let $c\modx 0N$. Then by the definition of $\phi_s$, we have $\Lambda_k\circ A=\frac{\phi_{\br{kd}}-\phi_{\br{d}}}{\phi_{\br{2d}}-\phi_{\br{d}}}$. From now on, to save labor, we put $r=\br{2d},s=\br{kd},t=\br d$. Then since $r,s,t$ are distinct from each other and $\min(s,t)=\min(r,t),~(d,N)=1$, we have $r,s,t\ne 0,N/2$ and $t<r,s$. We have only to prove $t=1$. Let us assume $t>1$.
Let $T=\begin{pmatrix}1&0\\1&1\end{pmatrix}$. Then 
\begin{equation}\label{eq5}
\Lambda_k\circ T=\left(\frac{\phi_s-\phi_t}{\phi_r-\phi_t}\right)\circ T.
\end{equation}
If $\ell$ is an integer such that $0<\ell<N/2$, then $\mu(\ell)=1,\br \ell=\ell$. Let $u=\zeta q$. Then
, 
\begin{equation}\label{eq6}
\phi_\ell[T]_2\equiv \sum_n nu^{\ell n}+u^{N-\ell} \mod q^N.
\end{equation}
From \eqref{eq5},
\[(\phi_r\phi_1+\phi_s\phi_2+\phi_t\phi_k)[T]_2=(\phi_t\phi_2+\phi_s\phi_1+\phi_r\phi_k)[T]_2.
\]
By comparing the order of $q$-series in the both sides, we see
$r=t+1<s$. Since $t\geq 2$ and $t+2\leq s<N/2$, we know that $2t\geq t+2,N>2t+4$. 
By \eqref{eq6} and by the inequality relations that $r=t+1,s\geq t+2,2t\geq t+2,N>2t+4$, we have modulo $u^{t+4}$,
\[ 
\begin{split}
&\phi_r\phi_1[T]_2\modx{u^{t+2}+2u^{t+3}}{u^{t+4}},~\phi_s\phi_2[T]_2\modx{0}{u^{t+4}},\\
&~\phi_t\phi_k[T]_2\modx{u^{t+k}}{u^{t+4}},\phi_t\phi_2[T]_2\modx{u^{t+2}}{u^{t+4}},\\
&\phi_s\phi_1[T]_2\modx{u^{s+1}}{u^{t+4}},~\phi_r\phi_k[T]_2\modx{0}{u^{t+4}}.
\end{split}
\]
Therefore we obtain a congruence:
\[2u^{t+3}+u^{t+k}\modx{u^{s+1}}{u^{t+4}}.\]
The coefficients of $u^{t+3}$ on both sides are distinct from each other, we have a contradiction. Hence $t=1$.
\end{proof}
We obtain the following theorem from the Gee-Stevenhagen theory in \cite{GA} and \cite{GAS}. See also Chapter 6 of \cite{SG}.
\begin{thm}
Let $N$ and $k$ be as above. Let $\alpha\in\H$ such that $\mathbf Z[\alpha]$ is the maximal order of an imaginary quadratic field $K$. Then the ray class field of $K$ is generated by $\Lambda_l(\alpha)$ over $\mathbf Q(\zeta,j(\alpha))$.
\end{thm}
\begin{proof}
The assertion is deduced from Theorems 1 and 2 of \cite{GA} and  Theorem \ref{th1}.
\end{proof}

\vspace{5mm}

{\small
\begin{tabular}{ll}
Faculty of Liberal Arts and Sciences \\
Osaka Prefecture University  \\
1-1 Gakuen-cho, Naka-ku Sakai\\
 Osaka, 599-8531 Japan\\
e-mail:\quad ishii@las.osakafu-u.ac.jp 
\end{tabular}
}
\end{document}